\documentclass[11pt,english]{article}

\usepackage[margin = 2.5cm]{geometry}

\usepackage{amsthm}
\usepackage{amsmath}
\usepackage{amssymb}
\usepackage{setspace}
\usepackage{mathtools}
\usepackage{graphicx}
\usepackage[hidelinks]{hyperref}
\usepackage{thm-restate}
\usepackage{cleveref}
\usepackage{hyperref}
\usepackage{enumitem}
\usepackage{framed}
\usepackage{subcaption}

\usepackage[square,sort,comma,numbers]{natbib}
\usepackage{floatrow}
\usepackage{comment}
\theoremstyle{plain}

\newtheorem*{thm*}{Theorem}
\newtheorem{thm}{Theorem}
\Crefname{thm}{Theorem}{Theorems}

\newtheorem*{lem*}{Lemma}
\newtheorem{lem}[thm]{Lemma}
\Crefname{lem}{Lemma}{Lemmas}

\newtheorem*{claim*}{Claim}

\crefname{claim}{Claim}{Claims}
\Crefname{claim}{Claim}{Claims}

\Crefname{prop}{Proposition}{Propositions}

\crefname{cor}{Corollary}{Corollaries}

\crefname{conj}{Conjecture}{Conjectures}

\newtheorem{qn}[thm]{Question}
\Crefname{qn}{Question}{Questions}

\Crefname{obs}{Observation}{Observations}

\Crefname{ex}{Example}{Examples}

\theoremstyle{definition}

\Crefname{prob}{Problem}{Problems}

\newtheorem{defn}[thm]{Definition}
\Crefname{defn}{Definition}{Definitions}

\theoremstyle{remark}

\renewenvironment{proof}[1][]{\begin{trivlist}
\item[\hspace{\labelsep}{\bf\noindent Proof#1.\/}] }{\qed\end{trivlist}}

\newcommand{\remove}[1]{}

\newcommand{\floor}[1]{
    \left\lfloor #1 \right\rfloor
}

\newcommand{\eps}{\varepsilon}

\newcommand{\cliq}[2]{$(#1,#2)$-cliquey}
\newcommand{\E}{\mathbb{E}}
\newcommand{\I}{\mathcal{I}}

\begin{document}


\title{Counting independent sets in structured graphs}

\author{
Matija Buci\'c\thanks{School of Mathematics, Institute for Advanced Study and Department of Mathematics, Princeton University, Princeton, USA. Email: \href{mailto:matija.bucic@ias.edu} {\nolinkurl{matija.bucic@ias.edu}}.} \and
Maria Chudnovsky \thanks{Princeton University, Princeton, USA. Supported by NSF-EPSRC Grant DMS-2120644 and by AFOSR grant FA9550-22-1-0083.} \and
Julien Codsi \thanks{Princeton University, Princeton, USA.}
}

\date{}

\maketitle

 \begin{abstract}
    \setlength{\parskip}{\smallskipamount}
    \setlength{\parindent}{0pt}
    \noindent
    Counting independent sets in graphs and hypergraphs under a variety of restrictions is a classical question with a long history. It is the subject of the celebrated container method which found numerous spectacular applications over the years. We consider the question of how many independent sets we can have in a graph under structural restrictions. We show that any $n$-vertex graph with independence number $\alpha$ without $bK_a$ as an induced subgraph has at most $n^{O(1)} \cdot \alpha^{O(\alpha)}$ independent sets. This substantially improves the trivial upper bound of $n^{\alpha},$ whenever $\alpha \le n^{o(1)}$ and gives a characterization of graphs forbidding of which allows for such an improvement. It is also in general tight up to a constant in the exponent since there exist triangle-free graphs with $\alpha^{\Omega(\alpha)}$ independent sets. We also prove that if one in addition assumes the ground graph is chi-bounded one can improve the bound to $n^{O(1)} \cdot 2^{O(\alpha)}$ which is tight up to a constant factor in the exponent.
 \end{abstract}

\section{Introduction}
Problems involving counting independent sets in graphs and hypergraphs have a long history. They have been studied both due to their intrinsic interest and since one can encapsulate many natural questions in terms of counting independent sets in an appropriate (hyper)graph see e.g.\ a recent survey \cite{samotij-couting-survey} for many examples and a detailed history of such questions. Let us denote by $\alpha(G)$ the maximum size of an independent set and by $i(G)$ the number of independent sets in a graph $G$. There are two trivial bounds which often serve as a baseline for more involved arguments.
\begin{equation}\label{eq:trivial}
    2^{\alpha(G)} \le i(G) \le \sum_{j=0}^{\alpha(G)} \binom{n}{j}. 
\end{equation}
 The lower bound follows since all subsets of the maximum size independent set are themselves independent and the upper bound simply account for all subsets of size up to $\alpha(G)$. Both of these bounds can be tight, for example, if $G$ is an empty or complete (hyper)graph, respectively. Note also that if $\alpha(G)=\Omega(n)$ the upper bound is also exponential in $\alpha(G)$ so the two bounds match up a constant in the exponent. We will be consequently mostly interested in the regime when $\alpha$ is somewhat small compared to $n$ when one can approximate the upper bound on the right by $\left(\frac{n}{\alpha(G)}\right)^{\alpha(G)}$ or even more loosely by just $n^{\alpha(G)}$.

Trying to improve the upper bound in \eqref{eq:trivial} has garnered a lot of interest over the years and is the subject of the celebrated container method. In the case of graphs it was introduced in the 1980's by Kleitman and Winston \cite{KW80,KW82} who used it to count lattices and graphs without four-cycles. Variations of the method have been used over the years to attack a variety of problems for example Alon \cite{noga-containers} and Sapozhenko \cite{sapozhenko-containers} used it to count the number of independent sets in regular graphs which in turn has applications to counting sum-free sets in Abelian groups (see e.g.\ \cite{noga-containers,sapozhenko-sum-free, sum-free}). Another remarkable example is the recent breakthrough lower bound on the off-diagonal Ramsey numbers \cite{MV}. The method has been extended to the hypergraph case independently by Balogh, Morris, and Samotij \cite{BMS} and Saxton and Thomason \cite{ST} and has found an even more impressive array of applications, see e.g.\ the survey \cite{containers-survey} produced to accompany an ICM 2018 talk on the subject of the container method for more examples and information. 
At a high level, the container method allows one to translate knowledge about a variety of statistics in a (hyper)graph (for example having a control of codegrees) to improvements on the upper bound in \eqref{eq:trivial}. 

In this paper, we study whether having structural information about a graph leads to an improvement in the upper bound \eqref{eq:trivial}. Perhaps one of the most studied structural properties of graphs is being $H$-free for a small fixed graph $H$. Here and throughout the paper a graph $G$ being \emph{$H$-free} stands for not having $H$ as an induced subgraph. This leads to a very natural question, does $G$ being $H$-free imply a substantial improvement in the upper bound \eqref{eq:trivial}? On the positive side, Farber in \cite{farber} showed in 1987 that forbidding a $2K_2$ does indeed lead to such an improved bound. Unfortunately, the answer is in general no. If one takes a graph consisting of $\alpha-1$ vertex disjoint complete graphs of order as equal as possible (namely the complement of the $(\alpha-1)$-partite Tur\'an graph on $n$ vertices) one obtains a graph with roughly $(n/\alpha)^{\alpha-1}$ independent sets, which is quite close to the upper bound in \eqref{eq:trivial}. On the other hand, this graph avoids most graphs as induced subgraphs, in fact, all of its induced subgraphs are themselves vertex disjoint unions of complete graphs. Our first result shows that one can improve the upper bound substantially if we forbid any such graph as an induced subgraph.

\begin{thm}\label{thm:H-free-intro}
    Any $bK_a$-free $n$-vertex graph $G$ with $\alpha=\alpha(G)$ has $i(G) \le n^{O(1)}\cdot \alpha^{O(\alpha)}.$
\end{thm}
Besides characterizing which forbidden graphs lead to an improvement this result is also tight up to a constant factor in the exponent so long as $a \ge 3$ and $b \ge 2$.  Indeed, taking $b-1$ disjoint complete graphs of order about $n/b$ each gives an $n$-vertex $bK_a$-free graph with roughly $(n/b)^{b-1}$ independent sets and independence number $b-1$. This shows the polynomial factor in $n$ needs to grow with $b$. One can show the same is true for $a$ by taking a graph consisting of $a$ disjoint complete graphs of order about $n/a$ each and then placing additional edges between parts independently with suitable probability to ensure there will be few copies of (not necessarily induced) $K_{a,a}$'s in the complement but with still many independent sets. One can then add a few edges to destroy these copies without destroying too many independent sets.  
More interestingly, one can not, in general, improve the second term either (beyond a constant factor in the exponent). This follows since there exist triangle-free graphs (so $bK_a$-free for any $a \ge 3$) with at least $\alpha(G)^{\Omega(\alpha(G))}$ independent sets. This follows by combining two results. The first one is due to Davies, Jenssen, Perkins and Roberts \cite{triangle-free-counts} (see also \cite{dhruv} for a slightly weaker but qualitatively similar result) showing that any $n$-vertex triangle-free graph with maximum degree $d$ has $i(G) \ge \exp\left(\left(\frac12+o_d(1)\right)\cdot \frac{\log^2 d}{d}\cdot n\right)$. The second is that there exist triangle-free graphs $G$ with all degrees being $(1+o(1))\sqrt{\frac12 n \log n}$ and $\alpha(G)\le(1+o(1))\sqrt{2n \log n}$. One obtains such a graph with high probability from the famous triangle-free process see e.g.\ \cite{bohman-keevash,fiz-pontiveros-morris} for more details on this topic. This produces a triangle-free graph with at least $\alpha^{(\sqrt{2}/4+o(1))\alpha}$ independent sets and $\alpha=\Theta(\sqrt{n \log n})$.

Our second main result shows that we can in fact go further and even match the trivial lower bound from \eqref{eq:trivial} up to a constant factor in the exponent if we in addition assume our ground graph is chi-bounded. Here, a hereditary class of graphs $\mathcal{G}$ is said to be chi-bounded if there exists some function $g:\mathbb{N} \to \mathbb{R}$ such that $\chi(G) \le g(\omega(G))$ for every $G \in \mathcal{G}$. It is a well-studied notion with many interesting applications and connections, see e.g.\ a recent survey of Scott and Seymour \cite{chi-boundedness}.

\begin{thm}\label{thm:chi-bounded-intro}
    Let $\mathcal{G}$ be a chi-bounded hereditary class of graphs. For any $bK_a$-free $n$-vertex graph $G \in \mathcal{G}$ with $\alpha=\alpha(G)$ we have $i(G) \le n^{O(1)} \cdot 2^{O(\alpha)}.$   
\end{thm}

Since forbidding $bK_2$ as an induced subgraph implies chi-boundedness, we conclude that the stronger bound of \Cref{thm:chi-bounded-intro} also holds in \Cref{thm:H-free-intro} when $a=2$, completing the picture as we have shown such an improvement is impossible when $a\ge 3$. 

There are two main tools behind our arguments. The first one is a certain hypergraph analogue of an \emph{induced} K\"ov\'ari-S\'os-Tur\'an Theorem introduced in the graph case by Loh, Tait, Timmons and Zhou \cite{induced-KST} and extended and used to settle a variety of problems recently in \cite{girao2023induced,axenovich2024induced,hunter2024k,bourneuf2023polynomial,illingworth2021induced}. The proof of this lemma is based on the dependent random choice technique (see e.g.\ the survey \cite{DRC-survey} for more information). The second ingredient is a certain local-to-global transference lemma for independent set counts, the proof of which is based on the ideas behind the container method.

\textbf{Notation.} Given a graph $G$ we denote by $\alpha(G),\omega(G),\chi(G)$ and 
$i(G)$ the independence number, clique number, chromatic number, and number of independent sets in $G$ respectively. We denote by $I_t(G)$ the family of all independent sets of size $t$ in $G$ and write $i_t(G)=|I_t(G)|$. Given a set of vertices $X$ in a graph $G$ we write $d_G(X)$ for the number of common neighbors of all vertices in $X$. For counting simplification purposes we consider an empty set of vertices as independent.
For the remainder of this paper, all logs are in base 2.

\section{Counting independent sets locally} 
In this section, we will prove our key technical tool, namely a hypergraph variant of the induced version of the K\"ov\'ari-S\'os-Tur\'an Theorem \cite{KST}. The classical theorem of K\"ov\'ari, S\'os, and Tur\'an dating back to 1954 states that if an $n$ vertex graph does not contain a $K_{s,s}$ as a not necessarily induced subgraph, then it has at most $O(n^{2-1/s})$ edges. 
It proved itself as an incredibly useful tool over the years as in many problems one can easily verify there are no $K_{s,s}$-subgraphs and as a result conclude that the graph in question is ``locally sparse''. Erd\H{o}s extended this result to hypergraphs in 1964 \cite{hypergraph-KST} showing that forbidding the $r$-partite $r$-uniform complete hypergraph as a not necessarily induced subgraph of an $r$-uniform $n$-vertex hypergraph implies the number of edges is at most $O(n^{r-\eps})$ for some $\eps>0$ depending on the forbidden hypergraph. Another natural extension of the classical theorem in which one only forbids $K_{s,s}$ as an \emph{induced} subgraph was only considered about 10 years ago in the graph case by Loh, Tait, Timmons, and Zhou \cite{induced-KST} and has found some very interesting further extensions and applications in the last few years. Unfortunately, the straightforward generalisation can not hold, as even a complete graph is $K_{s,s}$-free so long as $s \ge 2$. However, if one, in addition, forbids a clique (even of a size polynomial in $n$) one suddenly recovers the classical bound. 

Our technical lemma gives in a sense a common extension of both of these results. Roughly speaking our result states that if one forbids a complete $r$-partite graph as an induced subgraph and assumes there are no large cliques in our graph then there are at most $O(n^{r-\eps})$ cliques of size $r$ for some $\eps>0$ depending on the forbidden graph. Our result is actually slightly stronger, we assume a more flexible condition that every $m$ vertices contain an independent set of size $a$ (not having a large clique implies such a condition holds via Ramsey's Theorem).
We note that in \cite{induced-KST} a similar result focusing on the number of larger cliques has been proved but only under a much stronger assumption that a complete bipartite graph is forbidden. We also note that our result is not a full-fledged extension of Erd\H{o}s' result as it only applies to clique complexes, namely hypergraphs whose edges are cliques of a graph. One can actually prove such a stronger variant using a similar approach although since we do not need it here we choose not to do so. Part of the reason for this is that for our chi-bounded result, we need a very precise bound here, namely one that gives us a (slight) improvement even under a much weaker local assumption. We note also that the result is stated in the complement compared to the above discussion since this is how we will use it.

We start with a precise definition of the local condition we will use.
\begin{defn}
    A graph $G$ is \cliq{m}{a} if all $m$-vertex induced subgraphs of $G$ contain a $K_a$.
\end{defn}
So in particular, by Ramsey's theorem, we know that any graph $G$ is \cliq{\alpha(G)^a}{a} for any $a$.

We advise the reader that the following lemma and its proof might be initially easier to read under the assumption that $m$ is polynomially smaller than $n$ which ensures $\eps>0$ is an absolute constant depending only on $a$ and $b$. Indeed, this is sufficient for our proof of \Cref{thm:H-free-intro} and we suspect for most future applications as well. However, as mentioned above we need the more precise version, which allows for smaller, subpolynomial gains under weaker assumptions to prove \Cref{thm:chi-bounded-intro}.

\begin{lem}\label{lem:hypergraph-induced-KST}
    Let $a,b \ge 1$ and $n \ge m$ be integers. Let $G$ be a $bK_a$-free $n$-vertex graph which is \cliq{m}{a}. Then there are at most ${n^{b-\eps}}/{b!}$ independent sets of size $b$ in $G$, where 
    $$\eps:=\left(\frac{\log \frac nm}{8ab \log n}\right)^{b}.$$
\end{lem}
\begin{proof}
We write $\eps_b:=\left(\frac{\log \frac nm}{8ab \log n}\right)^{b}$ as the values of $m,a$ and $n$ for which we will use this expression always remain the same.
We prove the lemma by induction on $b$.  
If $b=1$, the lemma is vacuous since $G$ being \cliq{m}{a} and $K_a$-free imply we must have $n < m$. Let us now assume $b \ge 2$ and that the lemma holds for any $(b-1)K_a$-free graph. Let $G$ be a $bK_a$-free \cliq{m}{a} graph with $n$ vertices. If $a=1$ or $n<b$ there are no independent sets of size $b$ in $G$ and the lemma holds, so we may in addition assume $a \ge 2$ and $n \ge b$. 

We may assume $n > m$ as otherwise, $\eps_b=0$ and the desired bound is larger than $\binom{n}{b}$ so is trivially true. Suppose towards a contradiction that $G$ has more than $\frac{n^{b-\eps_b}}{b!}$ independent sets of size $b$. 
Our general strategy will be to find a set of vertices containing many independent sets of size $b-1$ for which all $a(b-1)$ subsets have more than $m$ non-neighbours. By induction, this set will contain $(b-1)K_a$ which will be used to create an induced $bK_a$.

Let $T$ be a random subset of vertices of $G$ obtained by sampling $t=\floor{\frac{4ab \log n}{\log \frac nm}}> 2ab\cdot \frac{\log n}{\log \frac nm}$ times uniformly at random, with repetitions, a vertex of $G$. Let $U$ be the set of vertices not adjacent to any vertex of $T$. 
Let $X$ count the number of independent sets of size $b-1$ contained in $U$. An independent set $I$ of size $b-1$ is contained within $U$ only if all $t$ vertices we sampled belong to its common non-neighbourhood, which happens with probability $(d_{\bar{G}}(I)/n)^{t}$ so by using Jensen's inequality (and convexity of $f(x)=x^t$) we get
\begin{align*}
    \E X = \sum_{I \in I_{b-1}(G)} \left(\frac{d_{\bar{G}}(I)}{n}\right)^t \ge |I_{b-1}(G)| \left(\frac{\sum_{I \in I_{b-1}(G)} 
 d_{\bar{G}}(I)}{n|I_{b-1}(G)|}\right)^t &\ge \frac{n^{b-1}}{(b-1)!} \left(\frac{b|I_b(G)|}{n\cdot \frac{n^{b-1}}{(b-1)!}}\right)^t\\ 
 &>\frac{n^{b-1}}{(b-1)!} \cdot n^{-t\eps_b}\\
 &> \frac{n^{b-1}}{(b-1)!}\cdot n^{-\eps_{b-1}/2},
\end{align*}

where in the second inequality we used $|I_{b-1}(G)|\le \binom{n}{b-1}\le \frac{n^{b-1}}{(b-1)!}$ (and the fact this term appears with a power $1-t \le 0$) and the hypergraph handshake lemma. In the third inequality, we used the assumed lower bound on $|I_b(G)|$. In the final inequality we used $t\eps_b= \floor{\frac{4ab \log n}{\log \frac nm}}\cdot  \left(\frac{\log \frac nm}{8ab \log n}\right)^{b}< \frac12 \left(\frac{\log \frac nm}{8a(b-1) \log n}\right)^{b-1}=\frac{\eps_{b-1}}2.$
On the other hand, given a set of $a(b-1)$ vertices with less than $m$ common non-neighbours the probability that this set is a subset of $U$ is at most $\left(\frac{m}{n}\right)^t$. So if we let $Y$ be the random variable counting the number of $a(b-1)$-sized subsets of $U$ with less than $m$ common non-neighbours we have 
$$\E Y \le \binom{n}{a(b-1)}\left(\frac{m}{n}\right)^t \le {n^{a(b-1)}}\cdot 2^{-t \log \frac nm}< {n^{a(b-1)}}\cdot n^{-2ab} ={n^{-a(b+1)}}.
$$
This shows that there is an outcome for which 
$$X-\binom{n}{b-1}\cdot Y > \frac{n^{b-1-\eps_{b-1}/2}-n^{-2-(a-1)(b+1)}}{(b-1)!}\ge \frac{n^{b-1-\eps_{b-1}/2}-n^{-2}}{(b-1)!}.$$ 
Let us consider such an outcome $U$. Note that since $X \le \binom{n}{b-1}$ and $n^{b-1-\eps_{b-1}/2}-n^{-2}>0$ we must have $Y=0$. Furthermore, since $b \ge 2$ we have $\eps_{b-1} \le 2(b-1)\frac{\log n}{\log \frac{n}{m}},$ which is equivalent to $n^{b-1-\eps_{b-1}/2}>m^{b-1}.$ This implies that $X>\frac{m^{b-1}-n^{-2}}{(b-1)!}>\binom{m-1}{b-1},$ so $|U| \ge m.$ This, together with $a \ge 2$ implies there must be an edge inside $G[U],$ so $X \le \binom{n}{b-1}-2\le \frac{n^{b-1}-2}{(b-1)!},$ where we used $n \ge b \ge 2$. Combined with the above lower bound on $X$ we get $n^{b-1}-2 > n^{b-1-\eps_{b-1}/2}-n^{-2}$ which implies $n^{\eps_{b-1}/2}>1+\frac{1}{n^{b-1}}$ which in turn gives 
$X> \frac{n^{b-1-\eps_{b-1}/2}-n^{-2}}{(b-1)!}\ge \frac{n^{b-1-\eps_{b-1}}}{(b-1)!}.$

If $b=2$ this implies $|U|=|X|>n^{1-\eps_1}> m$ so we can find a $K_a$ inside of $G[U]$. For $b \ge 3$ consider an auxiliary graph $G'$ on the same vertex set for which $G'[U]=G[U]$ but every vertex in $V(G) \setminus U$ is adjacent to every other vertex of $G'$. $G'$ is an $n$ vertex graph, is clearly \cliq{m}{a} and has more than $\frac{n^{b-1-\eps_{b-1}}}{(b-1)!}$ independent sets of size $b-1$. So by the inductive assumption, it must contain a $(b-1)K_a$ as an induced subgraph. Since $b-1>2$ the vertices of this $(b-1)K_a$ are not adjacent to all other vertices and hence must belong to $U$. Hence, in either case, we find an induced copy of $(b-1)K_a$ inside $G[U]$. Finally, since $Y=0$ the $a(b-1)$ vertices comprising this $(b-1)K_a$ have at least $m$ common non-neighbours. Among them, we can find a $K_a$ giving us an induced $bK_a$ in $G$ and the desired contradiction.
\end{proof}

\section{Translating counts from local to global}
In this section, we prove our second technical result which will allow us to propagate a tiny gain in the number of independent sets of some small size $b$ to a more substantial one globally. The basic idea behind the proof is reminiscent of the proofs of the container theorems (see e.g.\ \cite{BMS, ST}) although for the specific regime we work with we have a very simple proof (motivated in part by the ideas in \cite{jacob-paper}).

The following definition will come in useful in tracking independent set counts on subgraphs.
\begin{defn}
Given a graph $G$ and a real $m \ge 0$ let 
$$i(G,m):=\max_{\substack{G' \subseteq G, |G'| \le m}} i(G').$$
In other words, $i(G,m)$ denotes the maximum number of independent sets contained in an induced subgraph of $G$ with up to $m$ vertices.
\end{defn}

\begin{lem}\label{lem:recursion}
    Let $b \ge 1,$ and suppose $G$ is an $n$-vertex graph with at most $n^{b-\eps}/b!$ independent sets of size $b$. Then, $i(G) \le n^{b-1}\cdot i(G, n^{1-\eps/2^{b-1}})$.
\end{lem}
\begin{proof}
We call any independent set of size $b$ atypical. Now for any $1 \le i \le b-1$ we call an independent set of size $b-i$ \emph{typical} if it belongs to fewer than $n^{1-\eps/2^i}$ atypical independent sets of size $b-i+1$ and we say it is \emph{atypical} otherwise. Let $\I_t$ denote the collection of atypical independent sets of size $t$. So $\I_b$ consists of all independent sets of size $b$ in $G$ and has size $|\I_b| \le n^{b-\eps}/b!$. 
Since each atypical independent set of size $b-i$ belongs to at least $n^{1-\eps/2^i}$ atypical independent sets of size $b-i+1$ and each of these sets can contain at most $b-i+1$ of them we conclude $|\I_{b-i}|\cdot n^{1-\eps/2^i} \le |\I_{b-i+1}|\cdot (b-i+1)$. 
We conclude that
\begin{equation}\label{eq:num-typical}
    |\I_{b-i}| \le  \frac{|\I_{b}|\cdot b\cdot (b-1) \cdots (b-i+1)}{n^{1-\eps/2}\cdot n^{1-\eps/4}\cdots n^{1-\eps/2^i}}\le \frac{n^{b-i-\eps/2^i}}{(b-i)!}.
\end{equation} 

We now count the independent sets of $G$ based on the size of the largest typical independent set they contain. Note that there are at most $\binom{n}{b-i}i(G,n^{1-\eps/2^{i}})$ independent sets in $G$ with the largest typical independent set they contain having size $b-i$. Indeed, there are $\binom{n}{b-i}$ choices for the typical independent set, and once this is fixed the rest of the independent set is restricted to vertices which extend it into any atypical independent set of size $b-i+1$ (by maximality) and by definition of typicality there are at most $n^{1-\eps/2^i}$ such vertices. 
This only leaves the independent sets not containing any typical independent sets at all. Note that any such set is restricted to use only the vertices which are atypical independent sets of size one, of which there are by \eqref{eq:num-typical} at most $n^{1-\eps/2^{b-1}}.$ Putting all of this together we conclude that the number of independent sets in $G$ is at most 
\begin{equation*}
    i(G,n^{1-\eps/2^{b-1}})+\sum_{i=1}^{b-1} \binom{n}{b-i}i(G,n^{1-\eps/2^{i}}) \le {n}^{b-1}\cdot i(G,n^{1-\eps/2^{b-1}}),
\end{equation*}
as desired.
\end{proof}

\section{Counting independent sets in $H$-free graphs}
In this section, we prove our main results. We begin with \Cref{thm:H-free-intro} which we state here in a slightly more precise form.
\begin{thm}\label{thm:H-free}
    Let $a,b \ge 1$ be integers. There exists $C=C(a,b)\ge 0$ such that any $bK_a$-free $n$-vertex graph $G$ with $\alpha=\alpha(G)$ has $i(G) \le n^C\cdot \alpha^{2a\alpha}.$
\end{thm}

\begin{proof}
We will prove the theorem with $\eps=\eps(a,b):=(16ab)^{-b}$, $C=C(a,b)=(b-1)2^{b-1}/\eps$.
We proceed by induction on $n$. Observe first that if $n \le \alpha^{2a},$ then the desired inequality holds since $i(G)\le n^\alpha \le \alpha^{2a\alpha},$ as desired. Let us now assume that $n > \alpha^{2a}$ and that any induced subgraph $H$ of $G$ on $m<n$ vertices satisfies the desired inequality. In particular, this implies $i(G,m)\le m^C\cdot \alpha^{2a\alpha}$ for any $m <n$. 

Ramsey's Theorem implies that $G$ (and in fact any graph) is \cliq{\alpha^a}{a}.  \Cref{lem:hypergraph-induced-KST} implies there are at most $n^{b-\eps}/b!$ independent sets of size $b$ in $G$ since $$\left(\frac{\log \frac n{\alpha^a}}{8ab \log n}\right)^{b} \ge (16ab)^{-b}.$$ 
By \Cref{lem:recursion} this implies 
\begin{equation*}
    i(G) \le n^{b-1}\cdot i(G, n^{1-\eps/2^{b-1}}) \le {n}^{b-1}\cdot n^{C-C\eps/2^{b-1}}\cdot \alpha^{2a\alpha}= n^C\alpha^{2a\alpha}.
\end{equation*}
This completes the induction and the proof. 
\end{proof}
We note that by being more careful with the numbers in the above argument one can improve the bound to $n^{O(1)}\cdot \alpha^{(1+o(1))a\alpha}.$ 

We proceed with our result in the chi-bounded case, namely \Cref{thm:chi-bounded-intro}. The proof is similar to the above, the main distinction being that we can lower the base case to $n$ being linear in $\alpha$ rather than polynomial. This however comes with additional issues concerning the fact that \Cref{lem:hypergraph-induced-KST} stops giving us a polynomial gain in counts of small independent sets. The gain is still sufficient for our purposes though.

\begin{thm}\label{thm:chi-bounded}
    Let $a,b \ge 1$ be integers and let $\mathcal{G}$ be a chi-bounded, $bK_a$-free hereditary class of graphs. There exists $C=C(\mathcal{G})>0$ such that for every $G \in \mathcal{G}$ we have $i(G) \le |G|^{C} \cdot 2^{C\alpha(G)}.$    
\end{thm}

\begin{proof}
Let $g$ be a non-decreasing integral chi-bounding function, so that $\frac{|G'|}{\alpha(G')}\le \chi(G') < g(\omega(G'))$ for any $G' \in \mathcal{G}$. Let $C=g(a)\cdot (16ab)^{2b}.$

Let $G \in \mathcal{G}$ be an $N$-vertex $bK_a$-free graph and let $\alpha=\alpha(G)$.

Let $G'$ be an induced subgraph of $G$ with $\alpha g(a)$ vertices. By our chi-boundedness assumption we have $\alpha(G)g(\omega(G')) \ge \alpha(G')g(\omega(G'))> |G'|=\alpha(G)g(a)$ so $\omega(G')>a$. In particular, this implies that $G$, as well as any of its induced subgraphs, are \cliq{\alpha g(a)}{a}. 

Let $m:=\alpha g(a)$ and $\eps_n:=\left(\frac{\log \frac nm}{8ab \log n}\right)^{b}$. We note that while $\eps_n$ is a function of $a,b,m$ as well as $n$ the values of $a,b,m$ will remain fixed throughout the argument. Now for any $n \ge m$ \Cref{lem:hypergraph-induced-KST} implies that any $n$-vertex subgraph of $G$ contains at most $n^{b-\eps}/b!$ independent sets of size $b$. This in turn via \Cref{lem:recursion} implies it has at most $n^{b-1} \cdot i(G,n^{1-\eps_n/2^{b-1}})$ independent sets in total. Since this subgraph was arbitrary we have that for any real $n \ge m$ we have
\begin{equation}\label{eq:recursion}
    i(G,n) \le n^{b-1} \cdot i\left(G,n^{1-\eps_n/2^{b-1}}\right).
\end{equation}

Suppose first that $2m\le  n \le m^2$. Then, combined with \eqref{eq:recursion} we get
$$ \eps_n= \left(\frac{\log \frac nm}{8ab \log n}\right)^{b} \ge \frac{1}{(16ab)^b\log^b m}=:c \implies  i(G,n) \le n^{b-1} \cdot i(G,n^{1-c}).$$
The main benefit compared to \eqref{eq:recursion} is that the exponent $c$ does not depend on $n$ (which is also why we require an assumption on the rough size of $n$). This makes it easier to iterate the bound to get for any integer $j \ge 1$ that:
\begin{align*}
    i(n,G) &\le n^{b-1} \cdot i\left(G,n^{1-c}\right)\\
    &\le n^{(b-1)j} \cdot i\left(G,\max\{n^{(1-c)^j},2m\}\right)\\
    &\le n^{(b-1)/c} \cdot i(G,2m)\\
    &\le (m^2)^{(b-1)(16ab)^b\log^{b} m}  \cdot 2^{2m} \le 2^{2(b-1)(16ab)^b\log^{b+1} m +2m}\le 2^{2b(16ab)^b(b+1)^{b+1} m}\le 2^{C\alpha}. 
\end{align*}

Suppose now $n \ge m^2$. This combined with \eqref{eq:recursion} gives
$$ \eps_n= \left(\frac{\log \frac nm}{8ab \log n}\right)^{b} \ge \frac{1}{(16ab)^b}=:c \implies  i(G,n) \le n^{b-1} \cdot i(G,n^{1-c}).$$

Similarly to the above, we get that in this range
\begin{align*}
    i(G,n) &\le n^{(b-1)(1+(1-c)+(1-c)^{2}+\ldots+(1-c)^{j})} \cdot i\left(G, \max\{n^{(1-c)^{j+1}},m^2\}\right)\\
    &\le n^{(b-1)/c}\cdot i\left(G,m^2\right) \le n^{(b-1)(16ab)^b} \cdot 2^{C\alpha}.
\end{align*}

Putting all of this together, since $i(G)=i(G,N)$ this completes the proof if $N \ge 2m$. In the remaining case the trivial bound of $2^{N}\le 2^{2g(a)\alpha}$ suffices.
\end{proof}    

We note that by being more careful with the numbers in the above argument one can improve the bound to $n^{O(1)} \cdot 2^{(1+o(1))g(a)\alpha}$. 

    \section{Concluding Remarks and Open problems}
    In this paper, we improve the trivial upper bound on the number of independent sets in $bK_a$-free graphs. One of the main points of interest in improving upper bounds on the number of independent sets in a variety of graphs is that it allows for reducing the number of events we need to run various union bound arguments. It would be very interesting to find such applications of our results.

    Another interesting future direction might be to try to obtain similar improvements under other structural restrictions.

    Our bounds are tight up to a constant in the exponent in general. It would be interesting to obtain optimal exponents, at least asymptotically. This certainly seems to require additional ideas. With this in mind, we sketch here an alternative argument, inspired by the one of Farber \cite{farber} that he used to settle the $2K_2$-free case, which can be used to prove both of our main results in the case one forbids a disjoint union of $(b-1)K_2$ and $K_a$. Instead of counting all independent sets, we will only count maximal ones. Since each maximal independent set contains at most $2^{\alpha(G)}$ independent sets of $G$ and every independent set is contained in at least one maximal one this shows the difference between two counts is at most $2^{\alpha(G)}$. The argument now proceeds by induction on $b$. The case $b=1$ follows immediately from Ramsey's Theorem. For larger $b$, let us fix an arbitrary vertex $v$. There are three types of maximal independent sets in $G$
    \begin{enumerate}
        \item Ones that do not contain $v$, and are hence also maximal in $G \setminus \{v\}$.
        \item Ones that do contain $v$, and are, after removal of $v$, maximal in $G \setminus \{v\}$
        \item Ones that do contain $v$, but are, after removal of $v$, not maximal in $G \setminus \{v\}$
    \end{enumerate}
    It is easy to see that the number of maximal independent sets of type 1. plus the number of maximal independent sets of type 2. equals the number of maximal independent sets of $G \setminus \{v\}$. On the other hand, for any maximal independent set of type 3. there must exist a vertex $u$ adjacent to $v$ but otherwise having no neighbours in the maximal independent set. This means that we can upper bound the number of such maximal independent sets by going through all neighbours $u$ of $v$ (at most $n$ choices) and counting maximal independent sets in the set of common non-neighbours of $v$ and $u$. The crucial observation is that this set must be $(b-2)K_2+K_a$-free as we could extend any such induced subgraph by $vu$ to a $(b-1)K_2+K_a$. So we can use induction to get a strong upper bound on the number of independent sets of type 3. which suffices to prove the desired bounds. 
    
    As we already mentioned the number of independent sets and the number of maximal independent sets are at most a factor of $2^{\alpha(G)}$ apart, so they behave similarly. 
    On the other hand, no such relation seems to hold with the count of maximum independent sets. This leads to the natural question of whether one can improve our results if instead of counting independent sets we count only maximum ones. For example, even the following initial question remains open.
    \begin{qn}
    Does every $n$-vertex triangle-free graph with independence number $\alpha$ contain at most $2^{O(\alpha)}$ maximum independent sets?    
    \end{qn}
    If true this would be tight (up to a constant factor in the exponent) by considering $nK_2$. Or, more generally, by taking a disjoint union of $mK_2$ and a triangle-free graph on $n-2m$ vertices with as small independence number as possible so as to show the above result would be tight even for essentially any choice of $\alpha$.

     Given a graph $G$ with weights on its vertices, the Maximum Weight Independent Set (MWIS) problem asks to find the independent set in $G$ of maximum weight. It is well-known that MWIS is (very) computationally difficult in general, and is in particular (strongly) NP-hard \cite{mwis}. This motivated a considerable amount of work on getting efficient algorithms for MWIS for graphs under various restrictions, see \cite{DMS} for a more detailed treatment of the history and many examples. For example, it is known that a polynomial time algorithm exists for graphs with bounded treewidth. In fact, combining the results of \cite{DFGKM,DMS} gives that it is enough to have a bounded independence number in every bag of some tree decomposition\footnote{See \cite{chudnovsky2024tree} for definitions of a variety of concepts discussed in this paragraph.}. 
     A natural next step is to explore whether this can be further relaxed to allow the independence number bound to grow with $n$, in particular, given several recent results (see \cite{chudnovsky2024tree,DMS,DFGKM} and references therein) showing various structural restrictions imply the existence of a tree decomposition with each bag having independence number at most polynomial in $\log n$. In addition, there is a polynomial time algorithm for MWIS if we are given a tree decomposition in which every bag has only polynomially many independent sets. All of this motivates the question of exploring which structural restrictions guarantee that there are only polynomially many independent sets in a graph with a small independence number. \Cref{thm:H-free-intro} tells us that an $n$-vertex, $bK_a$-free graph $G$ with $\alpha(G)\le O(\log n / \log \log n)$ has $n^{O(1)}$ independent sets. \Cref{thm:chi-bounded-intro} tells us that under the additional assumption that $G$ belongs to a chi-bounded class of graphs the same holds already when $\alpha(G) \le O(\log n)$.

\textbf{Acknowledgments.} We would like to thank Nicolas Trotignon for useful discussions in the early stages of this project. We would also like to thank Wojciech Samotij for multiple useful comments and for suggesting one of the ideas behind the proof of \Cref{lem:recursion}. The first author would like to gratefully acknowledge the support of the Oswald Veblen Fund.

\providecommand{\MR}[1]{}
\providecommand{\MRhref}[2]{%
  \href{http://www.ams.org/mathscinet-getitem?mr=#1}{#2}
}


\providecommand{\bysame}{\leavevmode\hbox to3em{\hrulefill}\thinspace}
\providecommand{\MR}{\relax\ifhmode\unskip\space\fi MR }
\providecommand{\MRhref}[2]{%
  \href{http://www.ams.org/mathscinet-getitem?mr=#1}{#2}
}
\providecommand{\href}[2]{#2}

\end{document}